\documentclass[12pt]{amsart}
\usepackage{amssymb,amsmath,amsthm,comment}
\usepackage{setspace}
\newcommand{\shrinkmargins}[1]{
  \addtolength{\textheight}{#1\topmargin}
  \addtolength{\textheight}{#1\topmargin}
  \addtolength{\textwidth}{#1\oddsidemargin}
  \addtolength{\textwidth}{#1\evensidemargin}
  \addtolength{\topmargin}{-#1\topmargin}
  \addtolength{\oddsidemargin}{-#1\oddsidemargin}
  \addtolength{\evensidemargin}{-#1\evensidemargin}
  }
\shrinkmargins{0.5}

\newtheorem{theorem}{Theorem}
\newtheorem{lemma}[theorem]{Lemma}
\newtheorem{corollary}[theorem]{Corollary}

\newtheorem*{definition}{Definition}

\theoremstyle{remark}

\newtheorem*{remark}{Remark}

\numberwithin{theorem}{section} \numberwithin{equation}{section}

\newcommand{\R}{\mathbb{R}}
\newcommand{\C}{\mathbb{C}}

\newcommand{\Q}{\mathbb{Q}}

\newcommand{\Z}{\mathbb{Z}}
\newcommand{\N}{\mathbb{N}}

\newcommand{\G}{\Gamma}

\def\H{\mathbb{H}}

\begin{document}
\title{Maass Spezialschar of Level $N$}
%\author{Kathrin Bringmann}
%\address{School of Mathematics\\University of Minnesota\\ Minneapolis, MN 55455 \\U.S.A.}
%\email{bringman@math.umn.edu}  
\author{Bernhard Heim}
\address{German University of Technology in Oman, Muscat, Sultanate of Oman}
\email{bernhard.heim@gutech.edu.om}

%%%%\dedicatory{Preliminary}

%%\date{\today}
\begin{abstract}
In this paper the image of the Saito-Kurokawa lift of level $N$ with Dirichlet character
is studied. We give a new characterization of this so called Maass Spezialschar of level $N$ by
symmetries involving Hecke operators related to $\Gamma_0(N)$. We finally obtain for all prime numbers $p$
local Maass relations. This generalizes known results for level $N=1$.
\end{abstract}
\maketitle
\section{Introduction}
In 2012 \cite{Ib12}, T. Ibukiyama gave a systematic treatment 
of Saito-Kurokawa lifts of level $N$ with possible Dirichlet character.
First results in the classical setting had been obtained by B. Ramakrishnan, M. Manickham, and T. Vasudewa \cite{MRV93}.
In this paper, we study the image of the lifting, the Maass Spezialschar of level $N$. 
We obtain a new characterization by symmetries, generalizing previous work on liftings for the full 
Siegel modular group of degree two \cite{He10}. We refer to the original literature \cite{Ma79I, Ma79II, Ma79III,Ku78} and \cite{Za80} for
the Saito-Kurokawa conjecture and the Maass Spezialschar. An excellent introduction is given in \cite{EZ85}.
See also Oda's general viewpoint of theta lifts \cite{Od77}.

Let $F \in M_k^2(\Gamma_0^2 (N), \chi)$ be a Siegel modular form of 
Hecke type of integral weight $k$, degree $2$ and level $N$ with Dirichlet character $\chi$. Here $\chi(-1)= (-1)^k$.

Let $\Delta_N(l) $ be the set of all integral matrices $g = \left( \begin{smallmatrix} a & b \\ c & d \end{smallmatrix}\right)$ with
determinant $l$, with $N  \vert c$ and $(a,N)=1$. We put $\chi(g):= \overline{\chi(a)}$ and $\Gamma_0(N) = \Delta_N(1)$. Let $\vert_k$ be the Petersson slash operator and $\widetilde{g}^{\uparrow \downarrow}$ be two dual embeddings of $\Delta_N(l)$ into the symplectic group $Sp_2(\R)$. Then we have the following new characterisation 
of the Maass Spezialschar. The Siegel modular form $F \in M_k^2(\Gamma_0^2 (N), \chi)$  is a lift if and only if for all $l \in \N$:
\begin{equation}
%%\label{symmetrynew}
\sum_{g \in \Gamma_0(N) \backslash \Delta_N(l)}   \chi(g)^{-1} \, \left(F \vert_k \widetilde{g}^{\uparrow} \right) =  
\sum_{g \in \Gamma_0(N) \backslash \Delta_N(l)}   \chi(g)^{-1} \, \left(F \vert_k \widetilde{g}^{\downarrow} \right). \qquad \qquad (*_l).
\end{equation}
%%Comparing the Taylor, Fourier-Jacobi, Fourier expansion of both sides of $(*_l)$ lead to several applications.
The level one case was previously proven \cite{He10} by working out the relation of the 
Taylor expansion and properties of certain differential operators. 
In this paper we give a new and more simple proof by studying the Fourier-Jacobi expansion. 
This approach, involving well-known
properties of the Hecke algebra $\mathcal{H}(\Gamma_0(N), \Delta_N)$ (\cite{Mi06}) 
is more transparent and natural. Here $\Delta_N$ is the union of all $\Delta_N(l)$.
The Hecke algebra is commutative, zero-divisor free and decomposes in local components. Hence it is
sufficient to check the symmetries only locally, which leads finally to the result that $F$ is in the Maass Spezialschar iff
$(*_p)$ is satisfied for all primes $p$. Of course the symmetries degenerate if $p \vert N$. 
%%$(T(p), T(p,p)$ for $p\not\vert N$ and $T(p)$ otherwise. 
For further generalisation, note that the following identity 
in the Hecke algebra $\mathcal{H}(\Gamma_0(N), \Delta_N)$ is crucial.
\begin{equation}
T(m) \circ T(n) = \sum_{\substack{ d \vert (m,n) \\ (d,N) = 1}} d \,\, T(d,d) \,\,T \left(\frac{mn}{d^2} \right)
\end{equation}
The element $T(l)$ degenerates if $(l,N) >1$ (see Miyake \cite{Mi06}, Theorem 4.5.13 (i)).

The symmetries $(*_l)$ encode a new type of Maass relations for Saito-Kurokawa lifts of Hecke type.
Let $\mathcal{X}$ denote the set of
half-integral positive semi definite matrices $\left( \begin{smallmatrix} n & r/2 \\ r/2 & m \end{smallmatrix}\right)$.
Let $\mathcal{X}^{*}$ be the subset, where the zero matrix is removed.
%and $\mathcal{X}^{+}$ subset with the zero set removed and with all elements are positive definite.
We put $A(T)=0$ if $T \not\in \mathcal{X}$. Let $F \in M_k^2(\Gamma_0^2(N),\chi)$ with Fourier coefficients $A(T)= A(n,r,m)$.
Then $F$ is in the Maass space iff  for all $T \in  \mathcal{X}^{*}$ and $l \in \N$:
 \begin{equation}
 \sum_{d \vert (n,r,l)} d^{k-1} \, \chi(d) \,\, A\left( \frac{nl}{d^2}, \frac{r}{d},m\right)
 =
 \sum_{d \vert (l,r,m)} d^{k-1} \, \chi(d) \,\, A\left( n, \frac{r}{d},\frac{ml}{d^2}\right).
 \end{equation}
Here $(n,r,l)$ denotes the greatest common divisor. 
As a consequence we obtain the useful application that $F$ is in the Maass Spezialschar 
iff for all $T \in \mathcal{X}^{*}$ and for all prime numbers $p$
the following Maass $p$-relations are satisfied:
%%\begin{equation}
%%A(pn,r,m) + p^{k-1} \chi(p) A\left(\frac{n}{p}, \frac{r}{p},m \right) = 
%%A(n,r,pm) + p^{k-1} \chi(p) A\left(n,\frac{r}{p}, \frac{m}{p}\right).
%%\end{equation}
\begin{align}
A(pn,r,m) + p^{k-1} \chi(p) & A\left(\frac{n}{p}, \frac{r}{p},m \right) \\ = &
A(n,r,pm) + p^{k-1} \chi(p) A\left(n,\frac{r}{p}, \frac{m}{p}\right).\nonumber
\end{align}
This gives a significant generalization to the known Maass $p$-relations for $N=1$ (see the survey
\cite{FPRS13} for further background information).
Note $\chi(p) = 0$ iff $p\vert N$. For $p\vert N$ we have $A(pn,r,m)=A(n,r,pm)$.
In the literature (see \cite{EZ85}, \cite{Ib12}) the equivalent Maass relations are stated as
\begin{equation}
 A(n,r,m) = \sum_{d \vert (n,r,m)} d^{k-1} \, \chi(d) \,\, A\left( \frac{nm}{d^2}, \frac{r}{d},1\right)
 \end{equation}
for all $T \in \mathcal{X}^{*}$.

Recently \cite{HM15}, together with Murase, we had been able to use a multiplicative version of the
symmetry principle $(*_l)$ to give a characterization of holomorphic Borcherds lifts and a new proof 
of Bruiniers converse theorem for the discriminant kernel group. Borcherds proved that his lifts have certain special divisors and
Bruinier proved that if a form has these special divisors, then the form is a lift.
We refer to \cite{Br16} for recent developments. It would be interesting to transfer some of the results of this paper
to the theory of Borcherds lifts for congruence subgroups.
\section{Modular Forms}
For basic facts about elliptic modular forms and Hecke theory we refer to \cite{Sh71,Mi06}. For Siegel modular forms
especially of degree 2 we recommend \cite{Fr83}, \cite{An87} and \cite{EZ85} (also standard reference for Jacobi forms).
\subsection{Preliminaries and Basic Notations}
Let $R$ be a subring of the real numbers $\R$ and let $N,k,n,r,m$ usually denote integers. Let $\chi$ be a Dirichlet character modulo $N$.
We denote $e(Z):= \text{exp}(\text{trace}(Z))$ for every suitable matrix $Z$. The symplectic group $GSp^{+}(n,\R)$ of positive
similitudes of degree $n$ acts on the Siegel upper half space $\H_n$. 
Further
let $F$ be a holomorphic function on $\H_n$
and let $\gamma = \left(\begin{smallmatrix} A & B \\ C & D  \end{smallmatrix}\right) \in GSp^{+}(n,R)$ and $Z \in \H_n$. Then
\begin{eqnarray*}
\gamma(Z) & & := (AZ + B)(CZ + D)^{-1}\\
F\vert_k \gamma \,\, (Z) & & := \text{det}(CZ + D)^{-k} \,\, F(\gamma(Z))\\
\widetilde{\gamma} & & := \text{det}(\gamma)^{\frac{1}{2n}}\\
g^{\uparrow} & & :=
\left( \begin{array} {cccc}a & 0 & b & 0\\ 0 & 1 & 0 & 0 \\
c & 0 & d & 0 \\ 0 & 0 & 0 &1 \end{array} \right)\\
g^{\downarrow} & & :=
\left( \begin{array} {cccc}   1 & 0 & 0 & 0 \\
0 & a & 0 & b \\ 0 & 0 & 1 & 0 \\
0& c & 0 & d \end{array} \right), \qquad 
g = \left( \begin{matrix}
a & b\\
c & d
\end{matrix} \right) \in SL_2(\R)\\
\Delta_N  & &:=  \left\{ \alpha =
\left( \begin{matrix}
a & b\\
c & d
\end{matrix} \right) 
\in GL_2^{+}(\Q) \cap \Z^{2,2}
\,\, \vert 
\,\, (a,N)=1,\,\,
N \vert c, \,\, \text{det}(\alpha) > 0 \right\}\\
Sp(n,R) & & :=\{ \gamma \in GSp^{+}(n,R) \, \vert \, \, \text{det}(\gamma) =1 \}\\
\Gamma_0^{(n)}(N) & & := \left\{ 
\gamma
\in Sp(n,\Z) \,\, \vert \,\, C \equiv 0 \pmod{N} \right\}.
\end{eqnarray*}
Let $\gamma \in \Gamma_0^{(n)}(N)$, we extend $\chi$ by $\chi(\gamma) := \chi(\text{det}(D))$.
We identify $GL_2(R)^{+}$ with $GSp^{+}(1,R)$ and $SL_2(R)$ with $Sp(1,R)$, and drop the index $n=1$ for simplification.
In the case $n=2$ we also identify 
\begin{equation*}
 Z = \left(\begin{matrix} \tau_1 & z \\ z & \tau_2  \end{matrix}\right) \mbox{ with } (\tau_1,z,\tau_2).
\end{equation*}
We further put $$
\mathcal{X} := 
\left\{ T=
\left(\begin{matrix} n & r/2 \\ r/2 & m \end{matrix}\right) \,\, \vert \,\, n,r,m \in \Z, \,\, T \geq 0 \right\}. $$
Then
$\mathcal{X}^{*} := \mathcal{X} -  \left(\begin{smallmatrix} 0 & 0 \\ 0 & 0  \end{smallmatrix}\right)$ and 
$\mathcal{X}^{+} := \{  T \in \mathcal{X} \, \vert \, T >0\}$. We identify $T$ with $(n,r,m)$.
Note that for all $T$ with $\text{det}(T)=0$ there exists a $U \in SL_2(\Z)$ such that $T[U]:= U^t T U = (l,0,0)$ with $l \in \N_0$.
Let us further denote by $d \vert (n,m)$ or $d\vert(n,r,m)$ that $d$ divides the gcd of the involved numbers.
The condition $d\vert (0,0,0)$ is empty.
%%%%%%%%%%%%%%%%%%%%%
%%%%%%%%%%%%%%%%%%%%%
%%%%%%%%%%%%%%%%%%%%%
%%%%%%%%%%%%%%%%%%%%%
%%%%%%%%%%%%%%%%%%%%%
%%%%%%%%%%%%%%%%%%%%%
%%%%%%%%%%%%%%%%%%%%%
%%%%%%%%%%%%%%%%%%%%%
%%%%%%%%%%%%%%%%%%%%%
%%%%%%%%%%%%%%%%%%%%%
\subsection{Modular forms of level $N$}
\begin{definition} 
Let $k,N$ be natural numbers. Let $\chi$ a Dirichlet character modulo $N$. Let $\Gamma$ be a congruence subgroup of
%%$Sp(n,\Z)$.
$\Gamma_0^{(n)}(N)$. 
A holomorphic function $F$ on $\H_n$ is denoted Siegel modular form of weight $k$, 
degree $n$ and Dirichlet character $\chi$ with respect to $\Gamma$ if for all $g \in \Gamma$ the functional equation
\begin{equation}
F\vert_k g = \chi(g) \,\, F
\end{equation}
is satisfied. In the case $n=1$ we additionally have to propose that $F$ is regular at the cusps.
The space of these forms is denoted by $M_k^n(\Gamma, \chi)$.
\end{definition}
We refer briefly to the behavior of Saito-Kurokawa lifts at the cusps. 
%%(for cusp forms and the case $n=1$), since 
The main focus of this paper is the characterization of lifts independent of their Fourier expansion,
Although we consider the expansion at
infinity to some extent. 
%%We suggest to study this topic separately.
\begin{definition}
We denote by $S_k^n(\Gamma, \chi)$ the subspace of cusp forms. These are $F \in M_k^n(\Gamma, \chi)$ with
$F\vert_k \gamma$ vanishing at all boundaries. 
\end{definition}
See \cite{Fr83, Mi06} and also \cite{Ib12} for a more explicit version of the definition, 
guided by the Satake compactificaton. In a nutshell, let $F$ be holomorphic on 
$\H_n$ satisfying the functional equation for all elements of $\Gamma$. 
Let $V(Y_0):= \{ Y \in \R^{n,n}\, \vert \, Y \geq Y_0 > 0\}$ for $Y_0$ positive definite and $\Gamma_{\chi}$ kernel of $\chi$ on $\Gamma$.
Then $F \in M_k^n(\Gamma, \chi)$ iff $F\vert \gamma$ is bounded on $V(Y_0)$ for all 
$Y_0$ and $\gamma \in \Gamma_{\chi} \backslash Sp(n,\Z)$.
This property is always satisfied for $n>1$ (Koecher principle) and has only be checked for $n=1$.

Further $F \in S_k^n(\Gamma, \chi)$ iff $\Phi(F\vert_k \gamma) =0 $ 
for all $\gamma \in \Gamma_{\chi} \backslash Sp(n,\Z)$. Here $\Phi$ 
is the Siegel lowering operator. We refer to Freitag (\cite{Fr83}, chapter II, 
Satake compactification, see also section 3.1 and 3.2 \cite{Ib12}).\\
{\bf Remark.} Let $n=2$ then it is sufficient to check to cuspidality for the representatives of
$$\gamma \in \Gamma_{\chi} \backslash Sp(2,\Z) / C_{2,1}(\Z).$$
Here $C_{2,1}(\Z)$ is the subgroup of $Sp(2,\Z)$ with last row given by $(0 \, 0 \, 0\, 1)$.
\subsection{Fourier and Fourier-Jacobi expansion}
Let $\Gamma$ be a congruence subgroup of $Sp(2,\Z)$ containing 
$$\left\{ \left( \begin{matrix} 1_2 & S \\ 0 & 1_2 \end{matrix} \right) \, \vert \, S = S^t \in \Z^{2,2} \right\}.$$
Then $F \in M_k^2(\Gamma, \chi)$ has the Fourier expansion
\begin{eqnarray}
F(Z) & = & \sum_{ T \in \mathcal{X}} A(T) \,\, e(TZ)\\
 & = & \sum_{ (n,r,m) \in \mathcal{X}} A(n,r,m) \,\, e(n \tau_1 + r z + m \tau_2)
\end{eqnarray}
In the following we will also put $q_1= e(\tau_1), \zeta = e(z)$ and $q_2= e(\tau_2)$.
Note that $F$ is a cusp form then $A(T) = 0$ for all $T \in \mathcal{X}^{+}$. Note that the converse is not true.
The Fourier-Jacobi expansion of $F$ is given by
\begin{equation}
F(\tau_1,z,\tau_2) = \sum_{m=0}^{\infty}  F_m(\tau_1,z) \,\, q_2^m.
\end{equation}
Then $F_m$ is called the $m$-th Fourier Jacobi coefficient of $F$. It is a Jacobi form of weight $k$ and index $m$.
Note that $F_m$ is a Jacobi cusp form, if $F$ is a cusp form.
\subsection{Jacobi Group}
We consider the Jacobi group $G^J (R)$ as the semi-direct product of $GL_2^{+}(R)$ and the (additive written)
Heisenberg group $$H(R)= 
\left\{ h=(\mu,\lambda; \kappa) \,\, \vert \,\, \mu, \lambda, \kappa \in R \right\}$$ (see \cite{Ib12}, Section 2).
We consider $h^{0}= (\lambda,\mu)$ as a row vector.
%%Let $R \subset \R$. 
Then 
$$
G^J(R) := \left\{ (g,h) \,\, \vert \,\, g \in GL_2^{+}(R), \, h \in H(R) \right\}.$$
The explicit group operation is given by:
$$ (g_1,h_1) (g_2,h_2) = (g_1 g_2, \text{det}(g_2)^{-1} \, (h_1^{0}\,g_2, \kappa_1) + h_2).$$
We further define the following subgroups and monoids  of $G^J(\Z)$.
\begin{eqnarray*}
\Gamma_0(N)^J & := & \left\{ (g,h) \in G^J(\Z) \, \vert \, g \in \Gamma_0(N) \right \} \\
\Delta_N^J & := & \left\{ (g,h) \in G^J(\R) \, \vert \, g \in \Delta_N \mbox{ and }h \in H(\Z)      \right \}.
\end{eqnarray*} Let $\mathbb{H}^J:= \H \times \C$. 
%%and the 
%%symplectic group with similitudes $GSp^{+}(n,\R)$ of degree $n$ on the Siegel upper half space $\H_n$ of degree $n$ ($n=1,2$).\\ 
Let $\gamma = \left(\begin{smallmatrix} A & B \\ C & D  \end{smallmatrix}\right) \in GSp^{+}(n,R)$. 
Let $g= \left(\begin{smallmatrix} a & b \\ c & d  \end{smallmatrix}\right) \in GL_2^{+}(R)$ 
with $\text{det}(g)=l$ and $h = (\mu, \lambda; \kappa) \in H(R)$.
Let $f$ be a complex valued function on $\H^J$ and $F$ on $\H_n$.
Let $k,m \in \N_0$. 
%%Let $\chi$ a Dirichlet character modulo $N$.
\begin{eqnarray*}
%%%\chi(\gamma) & & := \chi(\text{det}(D)), \,\, \gamma \in \Gamma_0^{(n)}(N) \\
\widehat{f}(\tau_1,z,\tau_2)  & & := f(\tau_1,z) \,\, e(m \tau_2), \,\, \mbox{ for }(\tau_1,z,\tau_2) \in \H_2\\
\widehat{g} & &:=  \left(\begin{matrix} 
a & 0 & b& 0\\
0 & l & 0& 0\\
c & 0 & d& 0\\
0& 0 & 0& 1
\end{matrix}\right)\\
\widehat{h} & &:=  \left(\begin{matrix} 
1 & 0 & 0& \mu  \\
\lambda & l & \mu& \kappa   \\
0 & 0 & 1& - \lambda   \\
0& 0 & 0& 1
\end{matrix}\right)\\
\widehat{G}^J(R) & & := \left\{  \widehat{g}\,\, \widehat{h} \,\, \vert \,\, g \in GL_2^{+}(R), h \in H(R)\right\}.
\end{eqnarray*}
Obviously the map 
$\!\!\! \phantom{x}^{\wedge}$ 
is a group isomorphism between $G^J(R)$ and $\widehat{G}^J(R)$, 
where the semi-direct product structure can be
recovered. Let $(g_1,h_1), (g_2,h_2) \in G^J(\R)$. Then
$$\widehat{(g_1,h_1)} \widehat{ (g_2,h_2)} =  \left( \widehat{g_1}\widehat{g_2}\right)\,\,
\left( \widehat{g_2}^{-1}\widehat{h_1}\widehat{g_2}
\widehat{h_2}\right), 
%%\mbox{ where }
$$
$$\widehat{g_2}^{-1}\widehat{h_1}\widehat{g_2} \in \widehat{H(\R)}.$$
\subsection{Jacobi Forms of level $N$}
In this section we recall the definition of Jacobi forms
of level $N$ with Dirichlet character.
Let $f:\H \times \C \longrightarrow \C$ and $k,m \in \N_0$.  Let $g^J = (g,h) \in G^J(\R)$ 
with $\text{det}(g) = l$. Then we attach to $f$ the function $\widetilde{f}$ defined by
\begin{equation}
\left(  \widehat{f}\vert_k \widehat{g}^J\right) \,(\tau_1,z,\tau_2) = 
\widetilde{f}(\tau_1,z) \,\, e(-ml\tau_2).
\end{equation}
This leads to a canonical action of $G^J(\R)$ on $\H^J$
and the definition of Jacobi forms. 
This avoids explicit calculations and displays the essential properties directly.
\begin{definition}
Let $\Phi$ be a holomorphic function on $\H^J$. Let $k,m \in \N_0$. 
Let $\chi$ be a Dirichlet character modulo $N$. Let $\Gamma^J$ be a congruence subgroup of 
$G^J(\Z)$ with the same Heisenberg part.
We denote by $\Phi$ a Jacobi form of weight $k$ and index $m$ with character $\chi$ 
with respect to $\Gamma^J$ if
$\Phi$ satisfies:
\begin{itemize}
\item[(1)] $\Phi\vert_{k,m} g^J = \chi(g) \Phi, \qquad  \mbox{ for all } g^J= (g,h) \in \Gamma^J$
\item[(ii)] For any $g \in GL_2^{+}(\Q)$, $\Phi\vert_{k,m}g$ has the Fourier expansion
\begin{equation*}
\sum_{r,n \in \Q} c^g(n,r) \,\, q^n \, \zeta^r,
\end{equation*}
where $c^g(n,r)=0$ unless $4nm-r^2 \geq 0$.
\end{itemize}
We say $\Phi$ is a Jacobi cusp form if $c^g(n,r)=0$ unless $4nm-r^2 > 0$ is satisfied.
The space of Jacobi form is denoted by $J_{k,m}(\Gamma^J,\chi)$ and the subspace of cusp forms by $J_{k,m}^{\text{cusp}}(\Gamma^J,\chi)$ 
\end{definition}
\begin{remark}
The property(ii) needs only be checked for $g \in SL_2(\Z)$.  Here the sum is running over
$n \in h_g^{-1}\Z, \, r \in \Z$ with $c^g(n,r)=0$, unless $4nm-r^2 \geq 0$ (and $>$ for being a cusp form).
\end{remark}
\begin{remark}
Let $F$ be a cusp form for a congruence subgroup on $\H_2$. Then $F$ vanishes at every cusp. Equivalent the Fourier expansion at each cusp
has only support (parametrization of Fourier coefficients) at positive definite half-integral matrices. Hence at each cusp the
Fourier Jacobi coefficients are Jacobi cusp forms.
\end{remark}
Next we recall the definition of the index shift operator $V_{l,\chi}$ 
(see \cite{Ib12}, section 3) and finally define the Saito-Kurokawa lift.

\begin{definition}
For $\chi$ a Dirichlet character modulo $N$ and an element of $\Delta_N$:
\begin{equation}\label{extend}
\chi \left( \begin{smallmatrix}
a & b\\
c & d
\end{smallmatrix} \right) := \overline{\chi(a)}.
\end{equation}
\end{definition}
\begin{definition}
Let $k,N \in \N$ and let $\chi$ be a Dirichlet character modulo $N$. Let $ \Phi \in J_{k,m}(\Gamma_0(N)^J,\chi)$, $m \in \N_0$.
Then we define for all $l\in \N$ the index shift operator:
\begin{equation*}
V_{l,\chi} : 
J_{k,m}(\Gamma_0(N), \chi)^J      \longrightarrow        J_{k,ml}(\Gamma_0(N), \chi)^J
\end{equation*}
given by the explicit construction
\begin{eqnarray*}
V_{l,\chi}(\Phi) &:=& 
l^{k-1} 
\sum_{   
g \in \Gamma_0(N) \backslash \Delta_N(l)} \chi(g)^{-1} \,\, \Phi \vert_{k,m} g\\
&=&
l^{k-1} 
\sum_{   
g \in \Gamma_0(N) \backslash \Delta_N(l)}
\chi(a) \, (c\tau +d)^{-k} \,\, e^{-lm \frac{c\tau^2}{c\tau +d}} \,\, \Phi (g(\tau), \frac{lz}{c \tau +d} )
\\
& = & l^{k-1}  V_{l,\chi}^{\circ}(\Phi).
\end{eqnarray*}
Here $\left( \begin{smallmatrix}
a & b\\
c & d
\end{smallmatrix} \right)$ and $g(\tau) = \frac{a \tau + b}{c \tau +d}$.
\end{definition}
\begin{definition} 
Let $\chi$ be a Dirichlet character modulo $N$. Let $ \Phi \in J_{k,m}(\Gamma_0(N)^J,\chi)$.
Then $\mathcal{L}_{N,\chi}(\Phi)$ is called the Saito-Kurokawa lift of $\Phi$. It is defined by
\begin{equation}
\mathcal{L}_{N,\chi}(\Phi)
\left( \begin{matrix} \tau_1 & z \\ z & \tau_2 \end{matrix}\right)
 := c(0) f_{k,\chi}(\tau_1) +
\sum_{l=1}^{\infty} 
V_{l,\chi}(\Phi)(\tau_1,z) \,\, e(l\tau_2).
\end{equation}
Here $c(0)$ is the constant term of $\Phi$. For the definition of the Eisenstein series $f_{k,\chi}$ we refer to
\cite{Ib12}, section 3.2.
\end{definition}
Theorem 3.2 and Theorem 3.6 \cite{Ib12} states that $\mathcal{L}$ is a linear injective map to $M_k (\Gamma_0^J(N),\chi)$.
If $\Phi$ is a cusp form. Then $\mathcal{L}(\Phi)$ is a cusp form. 
The image of $\mathcal{L}$ is called Maass Spezialschar of level $N$.
%%%%%%%%%%%%%%%%%%%%%%%%%%%%%%
%%%%%%%%%%%%%%%%%%%%%%%%%%%%%%
%%%%%%%%%%%%%%%%%%%%%%%%%%%%%%
\section{Hecke Theory}
References: Krieg \cite{Kr90}, Miyake \cite{Mi06}, Shimura \cite{Sh71}.
Let $G$ be a group and $\Gamma$ a subgroup. Two subgroups are commensurable 
%%$\approx$ 
if the intersection has finite index in each of the two subgroups.
Let $\widetilde{\Gamma}$ be all elements $g \in \G$ such that 
$g \Gamma g^{-1}$ is commensurable with the subgroup $\Gamma$ itself.

Let $\Delta \subset G$ be a monoid and $\Xi$ a set of commensurable subgroups $\Gamma$ of $G$, such that
$\Gamma \subset \Delta \subset \widetilde{\Gamma}$. Let $R$ be a commutative ring with $1$. Then we denote 
by 
$$ \mathcal{H}_R (\Gamma, \Delta) := \left\{\sum_{\alpha \in \Delta} a_{\alpha} \, \Gamma \alpha \Gamma \,\,  \vert \,\,
a_{\alpha} \in R \mbox{ and } a_{\alpha}=0 \mbox{ for almost all } \alpha
\right\}$$
the free $R$-module generated double cosets. 
Let further $R[\Gamma \backslash \Delta]$ denote the free $R$-module 
generated by the $\Gamma \alpha$ cosets, where $\alpha \in \Delta$.

Next, let $\Delta$ act on a $R$-module $M$ by $m \mapsto m^{\alpha}$.
%%\begin{equation}
%% \alpha : \begin{cases}  M & \longrightarrow M  \\
%%                         m & \mapsto m^{\alpha}
%%                         \end{cases}
%%\end{equation}
Let $M^{\Gamma}$ be the submodule of $\Gamma$-invariant elements of $M$.
Let $\Gamma \alpha \Gamma =\sqcup_i \Gamma \alpha_i \in R[\Gamma \backslash \Delta]$ be the disjoint coset decomposition.
This identification leads to $\mathcal{H}_R (\Gamma, \Delta) = R[\Gamma \backslash \Delta]^{\Gamma}$. 
Note that $\mathcal{H}=\mathcal{H}_R (\Gamma, \Delta)$ acts on $M^{\Gamma}$ via 
$$m\vert \Gamma \alpha \Gamma := \sum_i m^{\alpha_i}.$$
Note that $m^{\alpha}$ is in general not invariant by $\Gamma$, but by $\alpha^{-1} \Gamma \alpha \cap \Gamma$.
Let $\widetilde{M}:= R[\Gamma \backslash \Delta]^{\Gamma}$. Then the action of $\mathcal{H}$ on 
$\widetilde{M}$
implies the following
multiplication of double cosets. Let $\Gamma \alpha \Gamma =\sqcup_i \Gamma \alpha_i $ and $\Gamma \beta \Gamma =\sqcup_i \Gamma \beta_i$.
Then 
\begin{equation}
\Gamma \alpha \Gamma \circ \Gamma \beta \Gamma := \sum_{\gamma} \Gamma \gamma \Gamma,
\end{equation}
where
$$
c_{\gamma} = \sharp \left\{ (i,j) \, \vert \, \Gamma \alpha_i \beta_j = \Gamma \gamma \right\}.
$$
\subsection{Representations of Hecke algebras}
We make the assumption that $G=GL_2^{+}(\R)$ and $\Gamma$ a Fuchsian group with finite character $\chi$.
Let $\mathcal{H}$ be the Hecke algebra attached to the Hecke pair $(\Gamma, \Delta)$.
We further assume that
\begin{itemize}
\item[(i)] $\chi$ can be extended to a character of $\Delta$ and

\item[(ii)] that for $\alpha \in \Delta$ and $\gamma \in \Gamma$ with $\alpha \Gamma \alpha^{-1} \in \Gamma$:\\
$\chi( \alpha \gamma \alpha^{-1}) = \chi(\gamma)$.
\end{itemize}
Let $\Xi$ be the set of all subgroups of $\Gamma$ of finite index. 
Let $k\in \Z$ be fixed. Let $\Gamma_1$ be any element of $\Xi$.
Let $M_k(\Gamma_1,\chi)$ be the vector space of holomorphic functions on $\H$ (and the cusps) satisfying:
\begin{eqnarray*}
\left( f\vert_k \gamma \right) (\tau) &:=& j(\gamma,\tau)^{-k} \, f(\gamma(\tau))\\
&=& \chi(\gamma) \, f(\tau) \qquad \mbox{for all } \gamma \in \Gamma_1.
\end{eqnarray*}
Then $\Delta$ acts on the $\Z$-module 
\begin{equation*}
 M := \bigcup_{ \Gamma_1 \,\in \,\, \Xi} M_k ( \Gamma_1, \chi)
\end{equation*}
by mapping $f \in M_k ( \Gamma_1, \chi)$ to an element $f^{\alpha} \in M_k ( \Gamma_1 \cap \alpha^{-1} \Gamma_1 \alpha, \chi)$:
$$ f \mapsto f^{\alpha}:= \overline{\chi(\alpha)} \, f\vert_k \alpha$$
(here we apply property (ii) from above).
Let $\Gamma \alpha \Gamma = \sqcup_i \Gamma \alpha_i$. Then the 
operation of of the Hecke algebra $\mathcal{H}$ on $M^{\Gamma}$ is given by
\begin{equation}
f\vert \Gamma \alpha \Gamma := \sum_i f^{\alpha_i}.
\end{equation}
This extends linearly to $h \in \mathcal{H}$ and called Hecke operators. 
%%Note that for the operation of the Hecke algebra the property (ii) is not needed. 
We refer to Miyake \cite{Mi06}, Remark 2.8.1 and 2.8.2 for a short discussion 
on elements of the Hecke algebra and Hecke operators.
%%%%%%%%%%%%%%%%%%%%%%%%%%
%%%%%%%%%%%%%%%%%%%%%%%%%%
%%%%%%%%%%%%%%%%%%%%%%%%%%
\subsection{ Structure of the Hecke Algebra $\mathcal{H}(\Gamma_0(N), \Delta_N)$} \ \\
Let $\chi$ be a Dirichlet character modulo $N$. We have extented (\ref{extend}) to
$\Delta_N$
in such a way that (ii) is satisfied.
%%\chi \left( \begin{smallmatrix}
%%a & b\\
%%c & d
%%\end{smallmatrix} \right) := \overline{\chi(a)}.$$
The Hecke theory applies to $G=GL_2^{+}(\R), \Delta = \Delta_N, \Gamma = \Gamma_0(N) 
\mbox{ and } R= \Z$.
Let $\mathcal{H} = \mathcal{H}(\Gamma_0(N), \Delta_N)$.
Let $[a,d]$ be the diagonal matrix
$\left( \begin{smallmatrix}
a & 0\\
0 & d
\end{smallmatrix} \right) $. Every double coset 
$$\Gamma_0(N)\, \alpha \,  \Gamma_0(N) = \Gamma_0(N)\, [a,d] \, \Gamma_0(N) =: T(a,b)$$
can be uniquely represented by a diagonal matrix $[a,d]$, where $(a,N) = 1, \, a \vert d, \mbox{ and } ad= \text{det}(\alpha)$.
Further let
\begin{equation}\label{special} T(l) = \sum_{\substack{ad =l,\, a \vert d,\\ (a,N) =1}}
T(a,d) = \bigsqcup_{\substack{ad =l, \,(a,N)=1,\\ b \mod{d}}} \Gamma_0(N) 
\left( \begin{smallmatrix}
a & b\\
0 & d
\end{smallmatrix} \right) = \Gamma_0(N) \backslash \Delta_N(l).
\end{equation}
Here we identified double cosets with elements in $$\Z[ \Gamma_0(N)\backslash \Delta_N]^{\Gamma_0(N)}.$$
Double cosets decompose in local components. Let $a_1\vert d_1$ and $a_2 \vert d_2$. Then 
\begin{equation}
T(a_1 a_2, d_1 d_2) = T(a_1,d_1) \circ T(a_2,d_2) \mbox{ if } (d_1,d_2)=1.
\end{equation}
The Hecke algebra is commutative and decomposes as a restricted tensor product in local Hecke algebras $\mathcal{H}_p$ for all prime numbers $p$.
$$ \mathcal{H} = \otimes_p \mathcal{H}_p,$$
where $\mathcal{H}_p$ is generated by $T(p)$ and $T(p,p)$ 
if $p \not| N$ and 
$T(p)$ otherwise. Hence for every $h \in \mathcal{H}_p$ with $(p\not| N)$ we have
$ h \in \Z[x,y]$, where 
\begin{eqnarray*}
x&= &T(p)= \Gamma_0(N) 
\left( \begin{smallmatrix}  p & 0\\  0 & 1
\end{smallmatrix} \right) 
\bigsqcup_{b \pmod{d}} \Gamma_0(N) 
\left( \begin{smallmatrix}
1 & b\\
0 & p
\end{smallmatrix} \right)\\
y &=& T(p,p) = \Gamma_0(N) 
\left( 
\begin{smallmatrix}
p & 0\\
0 & p
\end{smallmatrix} 
\right).
\end{eqnarray*}
Let $p\vert N$. Then $h \in \Z[T(p)]$, where $$T(p) = \bigsqcup_{b \pmod{d}} \Gamma_0(N) 
\left( \begin{matrix}
1 & b\\
0 & p
\end{matrix} \right) \mbox{ for } p \vert N.$$
We will transfer the result of \cite{Mi06}, Theorem 4.5.13 (1) to the theory of lifts.
\begin{theorem}\label{identity}
Let $m,n$ are natural numbers. Then we have the following identity 
in the Hecke algebra $\mathcal{H}(\Gamma_0(N), \Delta_N)$.
\begin{equation}
T(m) \circ T(n) = \sum_{\substack{ d \vert (m,n) \\ (d,N) = 1}} d \,\, T(d,d) \,\, T \left(\frac{mn}{d^2} \right).
\end{equation}
\end{theorem}
%%%%%%%%%%%%%%%%%%%%%%%%%%%%%%%%%%%%%%%%%%%%%%%%%%%%%%%%%%%%%%%%%%
%%%%%%%%%%%%%%%%%%%%%%%%%%%%%%%%%%%%%%%%%%%%%%%%%%%%%%%%%%%%%%%%%%
%%%%%%%%%%%%%%%%%%%%%%%%%%%%%%%%%%%%%%%%%%%%%%%%%%%%%%%%%%%%%%%%%%
%%\newpage
To apply the general theory define $\Xi$ to be the set of all subgroups of $\Gamma_0(N)$ and
\begin{eqnarray*}
M & & : =\bigcup_{ \Gamma \, \in \, \Xi} M_k (\Gamma, \chi).\\
%%f^{\alpha} & & \Delta_N \mbox{ operates on }
%%M: \,\,
%%f^{\alpha} := \overline{\chi (\alpha)}\,\, f\vert_k \alpha.
\end{eqnarray*}
Then $\Delta_N$ operates on $M$ by $f^{\alpha} := \overline{\chi (\alpha)}\,\, f\vert_k \, \alpha$.
The Hecke algebra operates on
$M^{\Gamma_0(N)}$. Actually it already operates on $M_k(\Gamma_0(N),\chi)$.
We are mainly interested in the operation of $T(l,l)$ and $T(l)$ on $M^{\Gamma_0(N)} $.
\begin{eqnarray*}
f & \mapsto & T(l,l)(f) = f^{[l,l]} = \overline{\chi(l)}\, l^{-k}\, f\\
f & \mapsto & T(l)(f) = \sum_{\substack{ ad =l, \,(a,N)=1 \\ b \pmod{d}}} 
\overline{\chi(a)}\, f\vert_k \left(\begin{matrix} a & b \\ 0 & d \end{matrix} \right).
\end{eqnarray*}
%%\newpage
We have the two Hecke algebras $\mathcal{H}:= \mathcal{H}( \Gamma_0(N)), \Delta_N)$ and
$\mathcal{H}^J:= \mathcal{H}( \widehat{\Gamma}_0^J(N)), \widehat{\Delta}_N^J)$. We are mainly
interested in the image of the embedding
\begin{equation*}
\iota: \mathcal{H} \hookrightarrow \mathcal{H}^J, 
\qquad
\Gamma_0(N) \alpha \Gamma_0(N) \mapsto
\widehat{\Gamma}_0(N)^J \,\,\widehat{\alpha}\,\, \widehat{\Gamma}_0(N)^J. 
\end{equation*}
This map respects the coset decomposition
\begin{equation*}
\bigsqcup  \Gamma_0(N)\,\,\alpha_i   \mapsto \bigsqcup \widehat{\Gamma}_0(N)^J\,\, \widehat{\alpha_i}.
\end{equation*}
Note that this property is implicitly used in the definition of the 
operator $V_{N,\chi}(l)$ in \cite{EZ85}, \cite{Ib12} (see also \cite{He99}, section 3). Let 
\begin{equation}
M^J :=
\bigcup_{m=0}^{\infty} \,\, \bigcup_{\Gamma \,  \in  \,\Xi}  \widehat{J}_{k,m}\left(\Gamma \ltimes H(\Z),\chi \right).
\end{equation}
Here $\Xi$ denotes the set of all congruence subgroup of $\Gamma_0(N)$.
Let $\alpha \in \Delta_N(l)$, then
$M^J \longrightarrow M^J, \Phi \mapsto {\Phi}^{\widehat{\alpha}}$, where
$\Phi \in J_{k,m} (\Gamma_0^J(N),\chi)$. Then
\begin{eqnarray*}
\Phi \vert \Gamma_0^J(N) \, \alpha \, \Gamma_0^J(N) & := & \sum_i \Phi^{\alpha_i} \in J_{k,ml} (\Gamma_0^J(N),\chi)\\
& = & \sum_i   \overline{\chi(\alpha_i)}   \,\,\Phi \vert_{k,m} \alpha_i.
\end{eqnarray*}
We frequently switch between $\Phi$ and $\widehat{\Phi}$ and consider $\alpha$ as element of 
$\Delta_N$, $\Delta_N^J$, and $\widehat{\Delta}_N$
accordingly. We make all the obvious identifications if clear from the context.
Note that cusp forms map to cusp forms.
Finally we perform the translation of the formula 
(\ref{identity} into the Hecke-Jacobi algebra. Note
that a priori it was not clear that this is possible, 
since the general Hecke-Jacobi algebra is not abelian and has zero divisors \cite{He99}.
Let $V^0(m)$ correspond to $T(m)$ and $V^0(d,d)$ if $(d,N)=1$ as elements of $\mathcal{H}^J$.
Then we obtain
inside the Hecke algebra $\mathcal{H}(\Gamma_0^J(N), \Delta_N)$ the important algebraic identity
\begin{equation}\label{core}
V^0(m) \circ V^0(n) = \sum_{\substack{ d \vert (m,n) \\ (d,N) = 1}} d \,\, V^0(d,d) \,\,V^0 \left(\frac{mn}{d^2} \right).
\end{equation}

%%%%%%%%%%%%%%%%%%%%%%%%%%%%%%%%%%%Einstieg%%%%%%%%%%%%%%%%%%%%%%%%%%%%%%%%%%%%%%%%%%%%%%%%%%%%%%%%%%%%%%%%%%%%%%%%%%%%%%%%%%%%%%%%%%%%%%
%%
\section{Main Results}
%%%%%%%%%%%%%  %%%%%%%%%%%%%%%%%%
The Maass Spezialschar of level $N$ is given by
\begin{equation}
M_k^{\text{Spez}} (\Gamma_0^2(N),\chi) := \left\{  \mathcal{L}_{N,\chi}(\Phi) \,\, 
\vert \,\, \Phi \in J_{k,1}(\Gamma_0^J(N), \chi) \right\}.
\end{equation}
The subspace of cusp form we denote by $S_k^{\text{Spez}}(\Gamma_0^2(N),\chi)$. A Siegel modular form
$F M_k^2(\Gamma_0(N)^J,\chi)$ is in the Maass Spezialschar iff all the Fourier coefficients of $F$ satisfy the
general Maass relations 
\begin{equation}\label{ibu}
 A(n,r,m) = \sum_{d \vert (n,r,m)} d^{k-1} \, \chi(d) \,\, A\left( \frac{nm}{d^2}, \frac{r}{d},1\right).
\end{equation}
See also \cite{Ib12}section 3.4 and the observations at the end of the proof of 
Theorem \ref{TheoremSymmetry}.
Our argument is the following. All Fourier coefficients $A(T)$, $T \in \mathcal{X}^{*}$ 
are determined by the first Fourier-Jacobi coefficient of $F$.
This is a Jacobi form of weight $k$ and index $1$ of level $N$ 
and the relations reflect exactly the definition of $\mathcal{L}_{N,\chi}$.

In this section we prove that $F$ is a Saito-Kurokawa lift iff $F$ satisfies symmetries $(*_l)$ for all $l \in \N$.
We state two applications. First, it is sufficient to check $(*_p)$ for prime numbers and second
we obtain symmetric Maass relations (of course equivalent to (\ref{ibu}). Combined we obtain local Maass $p$-relations
generalizing the known level $N=1$ case, discovered first by Pitale, Schmidt and the author \cite{FPRS13}.

%%%%%%%%%%%%%%%%%%%%%%%%%%%%%%%%%%%%%%%%%%%%%%%%%%%%%%%%%%%%%%%%%%%%%%%%%%%%%%%%%%%%%%%%%%%%%%%%%
%%%%%%%%%%%%%%%%%%%%%%%%%%%%%%%%%%%%%%%%%%%%%%%%%%%%%%%%%%%%%%%%%%%%%%%%%%%%%%%Hauptsatz%%%%%%%%%
%%%%%%%%%%%%%%%%%%%%%%%%%%%%%%%%%%%%%%%%%%%%%%%%%%%%%%%%%%%%%%%%%%%%%%%%%%%%%%%%%%%%%%%%%%%%%%%%%
\subsection{Maass Spezialschar and Symmetries}
Saito-Kurokawa lifts, elements in the Maass Spezialschar, can be characterized by symmetries.
Note that these symmetries $(*_l)$ for all $l \in \N$ make it possible to study Saito-Kurokawa lifts by
properties of the Hecke algebra $\mathcal{H}(\Gamma_0(N), \Delta_N)$ originally constructed to study elliptic modular forms.

\begin{theorem}\label{TheoremSymmetry}
Let $k$ and $N$ be positive integers. Suppose $\chi$ is a Dirichlet character modulo $N$ satisfying $\chi(-1)= (-1)^k$.
Let $F \in M_k(\Gamma_0^{(2)}(N), \chi )$ be a Siegel modular form of weight $k$, degree $2$, and level $N$ with Dirichlet character $\chi$.
Then $F$ is a Saito-Kurowaka lift if and only if $F$ satisfies for all $l \in \N$ the symmetry relation $(*_l)$ given by
\begin{equation}\label{symmetrynew}
\sum_{g \in \Gamma_0(N) \backslash \Delta_N(l)}   \chi(g)^{-1} \, \left(F \vert_k \widetilde{g}^{\uparrow} \right) =  
\sum_{g \in \Gamma_0(N) \backslash \Delta_N(l)}   \chi(g)^{-1} \, \left(F \vert_k \widetilde{g}^{\downarrow} \right). \qquad \qquad (*_l).
\end{equation}
\end{theorem}
%%%%%%%%%%%%  %%%%%%%%%%%%%%%%%%%
\begin{proof}
Note that $(*_l)$ is well-defined, since
\begin{equation*}
\chi(\gamma \, g) = \chi(g) = \chi(a)^{-1};  \qquad g=\left( \begin{matrix}
a & b\\
c & d
\end{matrix} \right) \in \Delta_N(l) \mbox{ and } \gamma \in \Gamma_0(N).
\end{equation*}
First we show that $(*_l)$ implies that the $l$-th Fourier-Jacobi (FJ) coefficients $F_l$ of $F$ satisfy 
$F_l = V_{l,\chi}(F_1)$. This implies that for $F \in S_k(\Gamma_0^{(2)}(N), \chi )$ all FJ coefficients are obtained by
$V_{l,\chi}(F_1)$, where $F_1 \in J_{k,1}^{\text{cusp}} (\Gamma_0^J(N),\chi)$. Hence $F = \mathcal{L}_{n,\chi}(F_1)$.
For the general case we refer to the end of this proof.
%%Applying the invertible operator $|_k A$ to both sides of the symmetry relation $(*_l)$ with
Let
$$ A: = \left( \begin{matrix}
1 & 0\\
0 & 1
\end{matrix} \right) \times 
\left( \begin{matrix}
\sqrt{l} & 0\\
0 & \sqrt{l}^{-1}
\end{matrix} \right), \mbox{ then } 
A\left( \begin{matrix} \tau_1 & z \\ z & \tau_2 \end{matrix} \right) =
 \left( \tau_1, \sqrt{l} z, l \tau_2 \right).
$$
We deform $(*_l)$ on both sides by $\vert_k A$. This breaks the symmetry of $(*_l)$.
Nevertheless the projective matrices $\widetilde{g}^{\uparrow \downarrow}$ become integral and the iff part of the Theorem still
remains. Let $(*_l^A)$ be given by
\begin{equation}\label{broken}
\sum_{g \in \Gamma_0(N) \backslash \Delta_N(l)}   \chi(g)^{-1} \, \left(F \vert_k \widetilde{g}^{\uparrow}A \right) =  
\sum_{g \in \Gamma_0(N) \backslash \Delta_N(l)}   \chi(g)^{-1} \, \left(F \vert_k \widetilde{g}^{\downarrow} A\right). \qquad \qquad (*_l^A).
\end{equation}
We calculate the left side of $(*_l^A)$. 
%%and obtain
%%$$\left(F \vert_k \widetilde{g}^{\uparrow}A \right) = l^k\,\,  F \vert _k \widehat{g}.$$
%%
Note that 
$$ \widetilde{g}^{\uparrow} \, A = \sqrt{l}^{-1} \left( g \times  \left( \begin{matrix}
l & 0\\
0 & 1
\end{matrix} \right) \right),$$
which implies that $\left(F \vert_k \widetilde{g}^{\uparrow}A \right) = l^{k} F \vert _k \widehat{g}$.
This action is compatible with the FJ expansion of $F$:
$$F(\tau_1,z,\tau_2) = \sum_{m=1}^{\infty} F_m(\tau_1,z) \, q_2^m \mbox{ with }  q_2 = e(\tau_2).$$
Finally we obtain for the $l$-th FJ coefficient of the left side of $(*_l^A)$ the expression
$$ l^k V_{l,\chi}^0 (F_1).$$
Next we determine the $l$-th FJ-coefficient of the right side of $(*_l^A)$. We fix for $\Gamma_0(N) \backslash \Delta_N(l)$
the special representation system (\ref{special}) and obtain:
\begin{eqnarray*}
\sum_{\substack {a,d \in \N;\,\, ad=l\\ b \pmod{d}}}
\left( \frac{d}{l}\right)^{-k} \,\, \chi(a) \, 
F \left( \tau_1, a z, a^2 \tau_2 +  \frac{b}{d} \right).
\end{eqnarray*}
The $l$-th of this expression is equal to
\begin{eqnarray*}
\sum_{a,d \in \N;\,\, ad=l}
\left( \frac{d}{l}\right)^{-k} \,\, \chi(a) \, 
F_{\frac{l}{a^2}}(\tau_1,az) \,\, 
\left( \sum_{b \pmod{d}} e\left( \frac{l}{a^2}\frac{b}{d}\right)\right),
%%F \left( \tau_1, a z, a^2 \tau_2 +  \frac{b}{d} \right).
\end{eqnarray*}
which simplifies to $l \,\, F_l$ (only the term $d=l$ contributes).
\\
\\
Conversely, assuming that $V_{l,\chi}(F_1) = F_l$ for all implies $(*_l)$ for all $l \in \N$
by applying a pure algebraic relation in a corresponding Hecke algebra.
We start in comparing the $m$-th Fourier Jacobi coefficients of both sides 
of $(*_l^A)$, where $m=l_1 \, l_2$ and $l_2 = l$.
We obtain for the left side:
\begin{eqnarray*}
l_2^k \,\, 
\sum_{\substack{a,d \, \in \, \N;\,\, ad=l\\ b \pmod{d}}}
\chi(a) \, 
F_{l_1}  \left( \frac{a \tau_1 +b}{d}, az, l_2 \tau_2 \right) \\
= l_2^k \,\, V_{l_2,\chi}^0 (F_{l_1}) = l_2^k \,\, l_1^{k-1}\,\, 
\left( V_{l_2,\chi}^0 \circ V_{l_1,\chi}^0\right) (F_1).
\end{eqnarray*} For the right side we obtain:
\begin{eqnarray*}
& &
\sum_{a,d \, \in \, \N;\,\, ad=l_2}
\left( \frac{d}{l_2}\right)^{-k} \,\, \chi(a) \, 
F_{\frac{m}{a^2}}(\tau_1,az) \,\, 
\left( \sum_{b \pmod{d}} e\left( \frac{m}{a^2}\frac{b}{d}\right)\right)\\
= & &
\sum_{ad=l_2, \,\,  a \vert l_1} d \,\, 
\left( \frac{d}{l_2}\right)^{-k} \,\, \chi(a) \, 
F_{\frac{m}{a^2}}(\tau_1,az) \\
= & &
 l_2 \sum_{a \vert (l_1,l_2) } a^{k-1} \,\, \chi(a) \, 
F_{\frac{m}{a^2}}(\tau_1,az) \\
= & &
 l_2 \left(l_1 l_2\right)^{k-1}
 \sum_{a \vert (l_1,l_2) } a^{1-k} \,\, \chi(a) \, 
V_{\frac{l_1 l_2}{a^2}\chi}^0(F_1)(\tau_1,az).
\end{eqnarray*}
The operator $V^0(a,a)$ is defined by 
\begin{eqnarray*}
V^0(a,a)(F)(\tau_1,z,\tau_2)  & :=  &  \chi(a) \,\, F\vert_k 
\widehat{\left( \begin{matrix} a & 0 \\ 0 & a \end{matrix} \right)}
(\tau_1,z,\tau_2)\\
& = &  \chi(a) \,\, a^{-k} F(\tau_1,az, a^2 \tau_2)
,
\end{eqnarray*} 
which
leads to an action on Jacobi forms. Hence the right side is equal to
\begin{equation}
l_2 \left(l_1 l_2\right)^{k-1} 
\sum_{\substack{a \vert (l_1,l_2)\\ (a,N)=1} } 
a \left( V^0(a,a) \circ V_{\frac{l_1 l_2}{a^2}\chi}^0\right) (F_1).
\end{equation}
Comparing the left and right side, we are left with showing the following identity
inside the Hecke algebra $\mathcal{H}(\Gamma_0^J(N), \Delta_N^J)$:
\begin{equation}
V^0(m) \circ V^0(n) = \sum_{\substack{ d \vert (m,n) \\ (d,N) = 1}} d \,\, V^0(d,d) \,\,V^0 \left(\frac{mn}{d^2} \right).
\end{equation}
This is pure algebraic and independent of the involved Jacobi forms and Fourier Jacobi expansions. This formula has been obtained
in section 3 on Hecke theory.

Finally we consider the case when $F$ is not necessarily a cusp form. 
Let $A(n,r,m)$ be the Fourier coefficients of $F$. Then $(*_l)$
implies that
\begin{equation}
\sum_{d \vert (n,l)}
d^{k-1} \, \chi(d) \, A \left( \frac{nl}{d^2},0,0 \right) = 
\sum_{d \vert \,\,  l}
d^{k-1} \, \chi(d) \, A \left( n,0,0 \right).
\end{equation}
Let $a(l):= A(l,0,0)$. Then we obtain
$$
a(l) = \left( \sum_{d \vert \,\, l}
d^{k-1} \, \chi(d) \right)  \,\, a(1).$$
All possible $a(0)$ such that
\begin{equation}
f(\tau) = \sum_{n=0}^{\infty} \,\, a(n) \,\, q^n \,\, \in \, M_k( \Gamma_0, \chi) 
\end{equation}
are classified in \cite{Ib12}. Hence $F$ is a Saito-Kurokawa lift in the sense of Ibukiyama.
%%%%%Comment all
\end{proof}
\subsection{Applications}
\begin{corollary}
Let $F \in M_k^2(\Gamma_0^{(2)}(N),\chi)$ with Fourier expansion
$$ F(\tau_1,z,\tau_2) = \sum_{ T =(n,r,m) \in \mathcal{X}}  A(n,r,m) \,\, q_1^n \, \zeta^r \,\, q_2^m.$$
Then the following properties are equal.
\begin{itemize}
\item[(i)] $F$ is a Saito-Kurokawa lift (also called Maass lift)
\item[(ii)] All the Fourier coefficients of $F$ satisfy:
\begin{equation*}
 A(n,r,m) = \sum_{d \vert (n,r,m)} d^{k-1} \, \chi(d) \,\, A\left( \frac{nm}{d^2}, \frac{r}{d},1\right).
 \end{equation*}
\item[(iii)]All the Fourier coefficients of $F$ satisfy for all $l \in \N$:
 \begin{equation*}
 \sum_{d \vert (n,r,l)} d^{k-1} \, \chi(d) \,\, A\left( \frac{nl}{d^2}, \frac{r}{d},m\right)
 =
 \sum_{d \vert (l,r,m)} d^{k-1} \, \chi(d) \,\, A\left( n, \frac{r}{d},\frac{ml}{d^2}\right).
 \end{equation*}
%% \item[(iv)]All the Fourier coefficients of $F$ satisfy for all primes number $p$:
 %%\begin{equation*}
%% \sum_{d \vert (n,r,p)} d^{k-1} \, \chi(d) \,\, A\left( \frac{np}{d^2}, \frac{r}{d},m\right)
%%
%% =
%% \sum_{d \vert (p,r,m)} d^{k-1} \, \chi(d) \,\, A\left( n, \frac{r}{d},\frac{mp}{d^2}\right).
%%\end{equation*}
\end{itemize}
\end{corollary}
\begin{proof}
The Maass lift (called also Saito-Kurokawa lift, see \cite{Ib12} Introduction) 
and the relations of Fourier coefficients is given in \cite{Ib12}, section 3.4.
For the readers convenience we recall the equivalence of (i) and (ii). 
Let $F$ be a Maass lift then (ii) is satisfied (Proposition 3.8, \cite{Ib12}). If (ii) is satisfied
then $F$ is uniquely determined by the first Fourier-Jacobi coefficient 
and all the other Fourier-Jacobi coefficients are the expected lifts in the setting of Jacobi forms.
\\
(iii) implies (i) by putting $m=1$ in formula (iii). Next we show that (i) implies (iii).
We have already proven that $F$ is a Maass lift if and only if
\begin{equation*}
\sum_{g \in \Gamma_0(N) \backslash \Delta_N(l)}   \chi(g)^{-1} \, \left(F \vert_k \widetilde{g}^{\uparrow} \right) =  
\sum_{g \in \Gamma_0(N) \backslash \Delta_N(l)}   \chi(g)^{-1} \, \left(F \vert_k \widetilde{g}^{\downarrow} \right). \qquad \qquad (*_l)
\end{equation*}
for all $l \in \N$. 
%%%
We fix for $\Gamma \backslash \Delta_N(l)$
the special representative system
\begin{equation}
\left\{
\left( \begin{array}{cc} a & b \\ 0 & d \end{array} \right)
\,\, \Big{|} \,\, a,b \in \N; \,\, ad=l;  \,\,(a,N)=1;\,\, b=0,1,\ldots, d-1 
\right\}.
\end{equation}
Note that $\chi(g)^{-1} = \chi(a)$. Then the left side of the equation $(*_l)$ is equal to
\begin{eqnarray*}
%%\sum_{g \in \Gamma \backslash \Delta_N(l)}   \chi(g)^{-1} \, \left(F \vert_k \widetilde{g}^{\uparrow} \right) (\tau_1,z,\tau_2)
%%& = &
\sum_{\substack {a,b \in \N;\,\, ad=l\\ b \pmod{d}}}
l^{\frac{k}{2}} d^{-k} \chi(a) \, 
F \left( \frac{a \tau_1 +b}{d}, l^{-\frac{1}{2}} a z, \tau_2 \right)
%%\\
%%& = &  \sum_{\substack {a,b,d; ad=l\\ b \pmod{d}}}
%%l^{\frac{k}{2}} d^{-k} \chi(a) \, 
%%F \left( \tau_1, l^{-\frac{1}{2}} a z, \frac{a \tau_2 +b}{d} \right)            \\
%%&=&
%%\sum_{g \in \Gamma \backslash \Delta_N(l)}   \chi(g)^{-1} \, \left(F \vert_k \widetilde{g}^{\downarrow} \right)(\tau_1,z,\tau_2).
\end{eqnarray*}
We consider the left side of $(*_l)$.
For convenience we map $z \mapsto \sqrt{l} \, z$ and keep in mind that 
$\sum_{ b \pmod{d}} e(n \frac{b}{d}) = d$ if $d \vert n$ and $0$ otherwise. We obtain
$$\sum_{n,r,m} l^{\frac{k}{2}} \sum_{\substack{ a,d \in \N \\ ad=l, \,\, d\vert n}} 
d^{-k+1} \, \chi(a) \, A(n,r,m) \, q_1^{\frac{an}{d}} \zeta^{ra} q_2^m.$$
This is equal to
$$ l^{1-\frac{k}{2}} \sum_{n,r,m} \sum_{ a \vert (n,r,l)} 
a^{k-1} \, \chi(a) \,\,
A  \left( \frac{nl}{a^2}, \frac{r}{a},m \right)  q_1^n \, \zeta^r \, q_2^m.$$
Since the left side of the relation $(*_l)$ is symmetric to the right side this leads to the proof.
\end{proof}
Actually one has to check the relations in (iii) only for $l$ prime numbers. This follows from the observation
\begin{corollary} Let $F$ be a Siegel modular form of level $N$.
Then $F$ is a Saito-Kurowaka lift if and only if $F$ satisfies the symmetry relation $(*_l)$ for all prime numbers $l$.
\end{corollary}
\begin{proof}
This follows from the results of section 3.2.
\end{proof}
Putting this together leads to
\begin{corollary}({\bf Maass $p$-relations})\\
Let $F$ be a Siegel modular form of level $N$. Then $F$ is a Saito-Kurokawa lift iff the Fourier coefficients of $F$ satisfy
\begin{equation}
A(np,r,m) + p^{k-1} \chi(p) \,\, A \left(  \frac{n}{p}, \frac{r}{p},m \right) =  
A(n,r,p m) + p^{k-1} \chi(p) \,\, A \left(  n, \frac{r}{p},  \frac{m}{p} \right). 
\end{equation}
for all prime numbers $p$. Note that for $p\vert N$ the relations degenerate to
$$ A(np,r,m) = A(n,r,p m).$$
\end{corollary}

{\bf Acknowledgements.} To be entered later.

 \end{document}